\documentclass[a4paper,10pt]{article}

\usepackage[english]{babel}
\usepackage[T1]{fontenc}
\usepackage{graphicx}
\usepackage{amsmath, amsfonts,amsthm, mathrsfs, amssymb}
\usepackage{verbatim}

\usepackage{enumerate}
\usepackage[latin1]{inputenc}
\usepackage{geometry}

\newcommand{\wtV}{\widetilde{V}}

\newcommand{\C}{{\mathbb C}}

\newcommand{\N}{{\mathbb N}}

\newcommand{\R}{{\mathbb R}}
\renewcommand{\S}{{\mathbb S}}

\newcommand{\Z}{{\mathbb Z}}

\newcommand{\cA}{{\mathcal A}}
\newcommand{\cB}{{\mathcal B}}
\newcommand{\cC}{{\mathcal C}}

\newcommand{\cH}{{\mathcal H}}
\newcommand{\cI}{{\mathcal I}}

\newcommand{\cL}{{\mathcal L}}
\newcommand{\cM}{{\mathcal M}}
\newcommand{\cN}{{\mathcal N}}
\newcommand{\cO}{{\mathcal O}}
\newcommand{\cP}{{\mathcal P}}

\newcommand{\cR}{{\mathcal R}}

\newcommand{\cT}{{\mathcal T}}

\newcommand{\cV}{{\mathcal V}}

\newcommand{\cZ}{{\mathcal Z}}

\renewcommand{\d}{\partial}

\newcommand{\norm}[1]{\left\| #1 \right\|}

\def\sleq{\lesssim}
\def\sgeq{\gtrsim}
\def\di{{\rm d}}
\newcommand{\dip}[2]{ \frac{\partial #1}{\partial #2} }

\newcommand{\sI}{\mathscr{I}}
\newcommand{\bj}{{\boldsymbol{j}}}
\newcommand{\bk}{{\boldsymbol{k}}}
\newcommand{\bl}{{\boldsymbol{l}}}

\newtheorem{theorem}{Theorem}[section]
\newtheorem{lemma}[theorem]{Lemma}
\newtheorem{corollary}[theorem]{Corollary}
\newtheorem{proposition}[theorem]{Proposition}
\newtheorem{definition}[theorem]{Definition}
\newtheorem{remark}[theorem]{Remark}

\title{A Nekhoroshev type theorem for the nonlinear Klein-Gordon equation with potential}

\author{
 S. Pasquali \footnote{Dipartimento di Matematica, Universit\`a degli Studi di Milano, Via Saldini 50, 20133 Milano. \newline
 \textit{Email: } \texttt{stefano.pasquali@unimi.it}}
}

\begin{document}

\maketitle

\begin{abstract}
We study the one-dimensional nonlinear Klein-Gordon (NLKG) equation 
with a convolution potential, and we prove that solutions 
with small analytic norm remain small for exponentially long times. 
The result is uniform with respect to $c \geq 1$, 
which however has to belong to a set of large measure. 
\end{abstract}

\section{Introduction} \label{intro}

In this paper we study the real one-dimensional nonlinear Klein-Gordon (NLKG) 
equation with a convolution potential on the segment,
\begin{align} \label{NLKGpot}
\frac{1}{c^2} \; u_{tt} \; - \; u_{xx} \; + \; c^2 \; u \; + V \ast u + f(u) &= 0,
\end{align}
with $c \in [1,+\infty)$, $x \in I:=[0,\pi]$, 
$f:\R\to\R$ an analytic function with a zero of order $3$ at the origin, 
in the case of Dirichlet boundary conditions. 

In this paper we show that the technique developed in \cite{faou2013nekhoroshev}
 allows us to deduce that solutions with small initial data that are analytic 
in a strip of width $\rho>0$ remain analytic in a strip of width $\rho/4$ for a 
timescale which is exponentially long with respect to the size of the initial 
datum; however, we have to assume that both the parameter $c$ 
and the coefficients of the potential belong to a set of large measure. \\

In \cite{pasquali2017longII} we proved an almost global existence result 
uniform with respect to $c \geq 1$ for the NLKG with a convolution potential. 
Furthermore, we deduced that for any $\delta>0$ any solution in $H^s$ 
with initial datum of size $\cO(c^{-\delta})$ remains of size $\cO(c^{-\delta})$ 
up to times of order $\cO(c^{\delta(r+1/2)})$ for any $r \geq 1$. 
Here we use normal form techniques in order to establish a result valid 
for exponentially long times, but we have to use analytic norms instead 
of Sobolev ones. 

The issue of long-time existence for small solutions of the NLKG 
equation on compact manifolds has been quite studied; see for example 
\cite{delort2004long}, \cite{bambusi2007almost}, \cite{delort2009long}, 
\cite{fang2010long}, \cite{fang2017almost} and \cite{delort2017long}. 
However, all results in the aforementioned papers rely on a nonresonance 
condition which is not uniform with respect to $c$. We also point out that 
the almost global existence for small solutions has been established only 
for the segment $[0,\pi]$ and in the case of Zoll manifolds, 
such as the multidimensional spheres $\S^d$, $d \geq 1$. \\

The proof combines the argument of \cite{faou2013nekhoroshev} for the NLS 
equation on the torus with a diophantine type estimate for the linear 
frequencies which holds uniformly for $c \geq 1$. We mention that a diophantine 
estimate for the frequencies uniform with respect to $c$ has been already used 
in \cite{pasquali2017longII} in order to prove the almost global existence.

A further aspect that would deserve future work is the study of 
Nekhoroshev estimates for the NLKG without potential. This is expected
to be a quite subtle problem, since for $c \neq 0$ the frequencies of
NLKG are typically non resonant, while the limiting frequencies are
resonant. \\ 

The paper is organized as follows.
In sect. \ref{results} we state the results of the paper, together with 
some examples and comments. In sect. \ref{HamSetting} we introduce the 
notations and the spaces which we use for our result. In sect. \ref{PolSec} 
we define a special class of polynomials. In sect. \ref{nonlinsec} we show 
that the nonlinearity appearing in the NLKG equation belongs to that class. 
In sect. \ref{nonresSec} we study the resonances of the system. 
In sect. \ref{BNFSec} we introduce the notion of normal form, and in the last 
section we prove the main theorem.

\section{Statement of the Main Results} \label{results}

In \eqref{NLKGpot} we assume that the potential has the form 
\begin{align} \label{coeffpot}
V(x) &= \sum_{k \geq 1} \; v_k \; \cos(kx).
\end{align}

\indent By using the same approach of \cite{faou2013nekhoroshev}, 
we fix a positive $s$, and for any $M>0$ we consider the probability space
\begin{align} \label{probspace}
\cV := \cV_{s,M} &= \left\{ (v_k)_{k\geq1} \; : \; v'_k \; := \; M^{-1} (1+|k|)^{s} v_k \in \left[-\frac{1}{2},\frac{1}{2}\right] \right\},
\end{align}
and we endow the product probability measure on the space of $(c,(v'_k)_k)$.

It is well known that \eqref{NLKGpot} is Hamiltonian with Hamiltonian 
\begin{align} \label{HamNLKGpot}
H(v,u) &= \int_I \frac{c^2 |v(x)|^2 + |u_x(x)|^2+c^2|u(x)|^2 + (V \ast u)(x)u(x) }{2} \di x + \int_I F(u) \di x,
\end{align}
where $v:=u_t/c^2$ is the momentum conjugated to $u$, and $F(u)$ is such that 
$\d_uF=f$. Consider now the Sturm Liouville problem 
\begin{align}
-\d_{xx} \phi_k + V \ast \phi_k &= \lambda_k \phi_k, \label{SturmLio}
\end{align}
with Dirichlet boundary conditions on $I$: 
it is well known that all the eigenvalues are distinct, 
that the solutions $(\phi_k)_{k \geq 1}$ of \eqref{SturmLio} 
given by $\phi_k(x)= \pi^{-1/2} \sin(kx)$ form an orthonormal basis of $L^2(I)$. 
It is useful to expand both $u$ and $v$ with respect to $(\phi_k)_{k \geq 1}$,
\begin{equation} \label{exp1}
\begin{cases}
u(t,x) &= \sum_{k \geq 1} q_k(t) \phi_k(x), \\
v(t,x) &= \sum_{k \geq 1} p_k(t) \phi_k(x), \\
\end{cases}
\end{equation}
and to introduce the following change of variables 
\begin{align}
\xi_k &:= \frac{1}{\sqrt{2}} \left[ q_k \left( \frac{ \sqrt{c^2+\lambda_k} }{c} \right)^{1/2}-i p_k \left( \frac{c}{ \sqrt{c^2+\lambda_k} } \right)^{1/2} \right], \; \; k \geq 1, \label{Fseries1} \\
\eta_k &:= \frac{1}{\sqrt{2}} \left[ q_k \left( \frac{ \sqrt{c^2+\lambda_k} }{c} \right)^{1/2}+i p_k \left( \frac{c}{ \sqrt{c^2+\lambda_k} } \right)^{1/2} \right], \; \; k \geq 1. \label{Fseries2}
\end{align}
where $\lambda_k = k^2 + v_k$. 
Indeed, in the coordinates $(\xi,\eta):=((\xi_k)_{k\geq1},(\eta_k)_{k \geq 1})$ 
the Hamiltonian \eqref{HamNLKGpot} takes the form 
\begin{align} \label{HamNLKGpotnew}
H(\xi,\eta) &= H_0(\xi,\eta)+N(\xi,\eta),
\end{align}
where
\begin{align}
H_0(\xi,\eta) &:= \sum_{k \geq 1} \omega_k \xi_k\eta_k, \label{H0} \\
N(\xi,\eta) &:= \int_I F \left( 
\sum_{k \geq 1} \left( \frac{c}{ \sqrt{c^2+\lambda_k} } \right)^{1/2} \frac{\xi_k+\eta_k}{\sqrt{2}} \phi_k(x) 
\right) \di x, \label{Nlin} 
\end{align}
where the linear frequencies $(\omega_k)_{k\geq1}$ of \eqref{HamNLKGpotnew} are 
given by 
\begin{align} \label{freq}
\omega_k := \omega_k(c) &= c \sqrt{c^2+\lambda_k} \; = \; c^2 \; + \frac{\lambda_k}{1+\sqrt{1+\lambda_k/c^2} } \\
&= \; c^2 \; + \frac{\lambda_k}{2} -\frac{\lambda_k^2}{2c^2} \frac{1}{(1+\sqrt{1+\lambda_k/c^2})^2 }.
\end{align}

Equation \eqref{NLKGpot} is a semilinear PDE locally well-posed in the energy 
space $H^1(I) \times L^2(I)$ (see \cite{nakanishi2011invariant}, ch. 2.1). 
Consider a local solution $(u(t),c^2 v(t))$ of \eqref{NLKGpot}: a standard 
computations shows that 
$(u,c^2 v) \in H^1(I) \times L^2(I)$ if and only if 
$(\xi,\eta) \in l_2^{1/2}(I) \times l_2^{1/2}(I)$ solves the following system 
\begin{equation} \label{FSys}
\begin{cases}
-i \dot\xi_k &= \omega_k \xi_k + \frac{\di N}{\di \eta_k}, \\
i \dot\eta_k &= \omega_k \eta_k + \frac{\di N}{\di \xi_k}.
\end{cases}
\end{equation}
Moreover, we will denote by $\psi$ and $\bar\psi$ the functions given by 
the following expansions with respect to the eigenfunctions $(\phi_k)_{k \geq 1}$,
\begin{align}
\psi(t,x) &:= \sum_{k \geq 1} \xi_k(t) \phi_k(x), \label{psi} \\
\bar\psi(t,x) &:= \sum_{k \geq 1} \eta_k(t) \phi_k(x). \label{barpsi}
\end{align}
It is easy to check that $(u,c^2v) \in H^1 \times L^2$ solve equation 
\eqref{NLKGpot} if and only if $(\psi,\bar\psi) \in H^{1/2} \times H^{1/2}$ 
solve the following equation,
\begin{align} \label{NLKGpot2}
-i\psi_t &= c (c^2-\Delta+\wtV)^{1/2}\psi + 
\frac{1}{\sqrt{2}} \left(\frac{c}{ (c^2-\Delta+ \wtV)^{1/2} }\right)^{1/2} \; 
f \left( \left(\frac{c}{ (c^2-\Delta+ \wtV)^{1/2} }\right)^{1/2} \frac{\psi+\bar\psi}{\sqrt{2}} \right),
\end{align}
where $\wtV$ is the operator that maps $\psi$ to $V \ast \psi$. \\

Now, for $\rho>0$ we denote by $\cA_\rho:=\cA_\rho(I,\C \times \C)$ the space of 
functions that are analytic on the complex neighborhood of $I$ given by 
\begin{align*}
I_\rho := \{ w=x+iy : x \in I, y \in \R, |y| < \rho \},
\end{align*}
and continuous on the closure of this strip. 
The space $\cA_\rho$, endowed with the following norm, 
\begin{align*}
|(\varphi_1,\varphi_2)|_\rho:= \sup_{w \in I_\rho} |\varphi_1(w)|+ |\varphi_2(w)|, \; \; \forall (\varphi_1,\varphi_2) \in \cA_\rho,
\end{align*}
is a Banach space.

Our main result is the following theorem:

\begin{theorem} \label{NekhThm}
Consider the equation \eqref{NLKGpot}. 
For any positive $\beta<1$ and for any $\rho>0$ the following holds: 
there exist $\gamma>0$, $\tau>0$, and a set 
$\cR_\gamma:=\cR_{\gamma,\tau,M} \subset \; [1,+\infty) \times \cV$ 
satisfying 
\begin{align*}
|\cR_{\gamma} \cap ([n,n+1] \times \cV)| &= \cO(\gamma) \; \; \forall n \in \N_0,
\end{align*}
such that for any 
$(c,(v_k)_k) \in ([1,+\infty) \times \cV)\setminus \cR_{\gamma,\tau,R}$ 
there exist $K>0$ and a sufficiently small $R_0>0$ such that, if
\begin{align*}
(\psi_0,\bar\psi_0) \in \cA_\rho, \; &\; |(\psi_0,\bar\psi_0)|_\rho=R<R_0,
\end{align*}
then the solution of \eqref{NLKGpot} with initial datum $(\psi_0,\bar\psi_0)$
exists for times $|t| \leq e^{-\sigma_\rho |\log R|^{\beta+1} }$, 
$\sigma_\rho=\min(1/8,\rho/4)$, and satisfies 
\begin{align}
|(\psi(t),\bar\psi(t))|_{\rho/4} &\leq K \; R, \; \; |t| \leq e^{-\sigma_\rho |\log R|^{\beta+1} }. \label{expest1}
\end{align}
Furthermore, we have that
\begin{align}
\sum_{k \geq 1} e^{\rho |k|} \left| |\xi_k(t)|-|\xi_k(0)| \right| \leq R^{3/2}, \; &\; |t| \leq e^{-\sigma_\rho |\log R|^{\beta+1} }. \label{expest2}
\end{align}
\end{theorem}

\begin{remark}
In finite dimension $n$, the standard Nekhoroshev theorem controls the dynamics 
over timescales of order $exp(- \alpha R^{-1/(\tau+1)})$ for some $\alpha>0$ and 
for some $\tau>n-1$ (see \cite{nekhoroshev1977exponential} and 
\cite{benettin1985proof}; see also \cite{lochak1992canonical} and 
\cite{poschel1993nekhoroshev} for a more direct proof in the convex and 
quasi-convex case respectively). 

In the infinite-dimensional context there are only few results, mainly due to 
Bambusi and P\"oschel in the one-dimensional case 
(see \cite{bambusi1998property}, \cite{bambusi1999nekhoroshev}, 
\cite{bambusi1999long}, \cite{poschel1999nekhoroshev} and \cite{bambusi2002long}), and by Faou and Grébert in the multidimensional case (see \cite{faou2013nekhoroshev}). 
In particular, in \cite{bambusi1999long} Bambusi proved a Nekhoroshev result 
for the one-dimensional NLKG: he was able to control the dynamics of analytic 
solutions in a strip on a timescale of order $\cO( e^{-\alpha |\log R|^{\beta+1} } )$ 
for some $\alpha>0$ and $\beta<1$ (which is the same timescale we cover in 
Theorem \ref{NekhThm}), but assuming that the parameter that appears in 
the equation ranges over a compact interval. 
\end{remark}

\begin{remark}
Actually, our result is slightly different from the one obtained by 
Faou ang Grébert in \cite{faou2013nekhoroshev}: indeed, while they proved that 
there exists a full measure set of potentials for which each solution of the 
NLS equation corresponding to an initial datum with small analytic norm remains 
small for exponentially long times, in our result 
we prove that for ``most'' of the values of speeds of light and potentials 
each solution of \eqref{NLKGpot} corresponding to an initial datum with small 
analytic norm remains small for exponentially long times. 

Such a difference is motivated by the non resonance condition we prove below 
(see Theorem \ref{nonrescondthm}), which is an adaptation of the uniform 
diophantine estimates reported in \cite{pasquali2017longII}.
\end{remark}

By exploiting the same argument used to prove Theorem \ref{NekhThm} 
one can immediately deduce the following stability result 
for solutions with small (with respect to $c$) initial data.

\begin{corollary}
Consider the equation \eqref{NLKGpot} and fix $\delta>0$. 
Then then there exist $\gamma>0$, $\tau>1$, and a set 
$\cR_\gamma:=\cR_{\gamma,\tau,M} \subset \; [1,+\infty) \times \cV$ 
satisfying 
\begin{align*}
|\cR_{\gamma} \cap ([n,n+1] \times \cV)| &= \cO(\gamma) \; \; \forall n \in \N_0,
\end{align*}
such that for any 
$(c,(v_k)_k) \in ([1,+\infty) \times \cV)\setminus \cR_{\gamma,\tau}$, 
for any positive $\beta<1$ and for any $\rho>0$ the following holds: 
there exist $K>0$ and $c^\ast>0$ such that, if $c>c^\ast$ and
\begin{align*}
(\psi_0,\bar\psi_0) \in \cA_\rho, \; &\; \|(\psi_0,\bar\psi_0)\|_\rho=\frac{1}{c^\delta},
\end{align*}
then the solution of \eqref{NLKGpot} with initial datum $(\psi_0,\bar\psi_0)$ 
the solution of \eqref{NLKGpot} with initial datum $(\psi_0,\bar\psi_0)$ 
exists for times $|t| \leq e^{-\sigma_\rho |\delta \log c|^{\beta+1} }$, 
$\sigma_\rho=\min(1/8,\rho/4)$, and satisfies 
\begin{align}
\|(\psi(t),\bar\psi(t))\|_{\rho/4} &\leq \frac{K}{c^\delta}, \; \; |t| \leq e^{-\sigma_\rho |\delta \log c|^{\beta+1} }. \label{expest1cor}
\end{align}
\end{corollary}

\section{Hamiltonian setting} \label{HamSetting}

In the following we set $\cZ:= \N_0 \times \{\pm 1\}$, $\N_0=\{1,2,3,\ldots\}$; 
for any $j=(k,\delta) \in \cZ$ we define $|j|:=|k|$, and we denote by $\bar j$ 
the index $(k,-\delta)$. We also identify a couple $(\xi,\eta)$ with 
$(z_j)_{j \in \cZ}$, where 
\begin{align}
j=(k,\delta) \in \cZ &\rightarrow 
\begin{cases}
z_j= \xi_k & \text{if } \delta=1,\\
z_j= \eta_k & \text{if } \delta=-1.
\end{cases}
\end{align}
By a slight abuse of notation we will often denote by $z=(z_j)_{j \in \cZ}$ 
such an element. The above system \eqref{FSys} may be regarded as an 
infinite-dimensional Hamiltonian system with coordinates 
$(\xi_k,\eta_k)_{k \geq 1} \in \C^\cZ \times \C^\cZ$ and
symplectic structure 
\begin{equation} \label{symplform}
i \sum_{k \geq 1} \di\xi_k \wedge \di\eta_k.
\end{equation}

For a given $\rho>0$, we also consider the Banach space $\cL_\rho$ whose elements are those $z \in \C^\cZ$ such that 
\begin{align*}
\norm{z}_\rho &:= \sum_{j \in\cZ} e^{\rho|j|} |z_j| \; < \; +\infty.
\end{align*}
We say that $z \in \cL_\rho$ is \emph{real} if $z_{\bar j} = \bar{z_j}$ 
for any $j \in \cZ$.

\begin{lemma} \label{anlemma}
Let $\psi$, $\bar\psi$ be complex valued functions analytic on a neighborhood 
of $I$, and let $(z_j)_{j \in \cZ}$ be the sequence defined by \eqref{psi} and 
\eqref{barpsi}. Then for all $\mu < \rho$ we have 
\begin{enumerate}
\item if $(\psi,\bar\psi) \in \cA_\rho$, then $z \in \cL_\mu$, and 
$\|z\|_\mu \leq K_{\rho,\mu} |(\psi,\bar\psi)|_\rho$; \\
\item if $z \in \cL_\rho$ then $(\psi,\bar\psi) \in \cA_\mu$, and 
$|(\psi,\bar\psi)|_\rho \leq K_{\rho,\mu} \|z\|_\mu$,
\end{enumerate}
where $K_{\rho,\mu}$ is a positive constant depending only on $\rho$ and $\mu$.
\end{lemma}

\begin{proof}
Assume that $(\psi,\bar\psi) \in \cA_\rho$: then by Cauchy formula we have that 
for all $j \in \cZ$ 
\begin{align*}
|z_j| &\leq e^{-\rho|j|} |(\psi,\bar\psi)|_\rho.
\end{align*}
Hence for $\mu < \rho$ we have 
\begin{align*}
\|z\|_\mu \leq |(\psi,\bar\psi)|_\rho \sum_{j \in \cZ} e^{(\mu-\rho) |j|} &\leq 
|(\psi,\bar\psi)|_\rho \; 2 \sum_{k \geq 1} e^{(\mu-\rho) k} 
\leq \frac{ 2 e^{\mu-\rho} }{1-e^{\mu-\rho}} |(\psi,\bar\psi)|_\rho.
\end{align*}
Conversely, assume that $z \in \cL_\rho$: then 
$|\xi_k| \leq e^{-\rho k} \|(\psi,\bar\psi)\|_\mu$, and thus by \eqref{psi} we get 
that for all $x \in I$ and $y \in \R$ with $|y| \leq \mu$ 
\begin{align*}
|\psi(x+iy)| + |\bar\psi(x+iy)| \leq \sum_{k \in \cZ} |z_k| e^{k|y|} &\leq 
\|z\|_\rho \; 2 \sum_{k \geq 1} e^{-(\rho-\mu)k} 
\leq \frac{ 2 e^{\mu-\rho} }{1-e^{\mu-\rho}} \|z\|_\rho.
\end{align*}
hence $\psi$ and $\bar\psi$ are bounded on the strip $I_\mu$.
\end{proof}

Consider a function $G \in C^1(\cL_\rho,\C)$ we define its Hamiltonian vector 
field by $X_G:= J \nabla G$, where $J$ is the symplectic operator on $\cL_\rho$ 
induced by the symplectic form \eqref{symplform}, and 
$\nabla G(z) := \left( \frac{\d G}{\d z_j} \right)_{j \in \cZ}$, 
where for $j=(k,\delta) \in \cZ$ 
\begin{align*}
\frac{\d G}{\d z_j} &=
\begin{cases}
\frac{\d G}{\d \xi_k} & \text{if} \; \; \delta=1, \\
\frac{\d G}{\d \eta_k} & \text{if} \; \; \delta=-1.
\end{cases}
\end{align*}

\begin{definition}
For a given $\rho>0$ we denote by $\cH_\rho$ the space of real Hamiltonians $G$ 
satisfying
\begin{align*}
G &\in C^1(\cL_\rho,\C), \\
X_G &\in C^1(\cL_\rho,\cL_\rho).
\end{align*}
Let $G_1$, $G_2 \in H_\rho$, then we define the Poisson bracket between 
$G_1$ and $G_2$ via the formula 
\begin{align} \label{PoiBr}
\{G_1,G_2\} := \nabla G_1^T \; X_{G_2} &= i \sum_{k \geq 1} \frac{\d G_1}{\d \eta_k} \frac{\d G_2}{\d \xi_k} - \frac{\d G_1}{\d \xi_k} \frac{\d G_2}{\d \eta_k}.
\end{align}
\end{definition}

We say that the Hamiltonian $H$ is \emph{real} if $H(z)$ is real for all real $z$.

We associate to a given Hamiltonian $H \in \cH_\rho$ the corresponding Hamilton 
equations,
\begin{align*}
\dot z &= X_H(z) = J \; \nabla H(z),
\end{align*}
equivalently
\begin{align*}
\begin{cases}
\dot \xi_k &= -i \frac{\d H}{\d \eta_k}, \\
\dot \eta_k &= i \frac{\d H}{\d \xi_k}.
\end{cases}
&\; \; k \geq 1.
\end{align*}

We also denote by $\Phi^t_H(z)$ the time-$t$ flow associated with the previous 
system. We just remark that if $z=(\xi,\bar\xi)$ and if $H$ is real, then also 
the flow $\Phi^t_H(z) =(\xi(t),\eta(t))$ is real for all $t$, namely 
$\xi(t)=\bar\eta(t)$ for all $t$. Moreover, the Hamiltonian 
\begin{align*}
H(\xi,\eta) &= \sum_{k \geq 1} \omega_k \xi_k\eta_k + N(\xi,\eta),
\end{align*}
with $N$ given by \eqref{Nlin}, leads to the system \eqref{FSys}, that is the 
NLKG equation \eqref{NLKGpot} written in the coordinates $(\xi,\eta)$.

\begin{remark}
We point out that the Hamiltonian $H_0=\sum_{k \geq 1} \omega_k\xi_k\eta_k$, which 
corresponds to the linear part of \eqref{HamNLKGpotnew}, does not belong to 
$\cH_\rho$. However, it generates a flow which maps $\cL_\rho$ to $\cL_\rho$, and 
it is explicitly given by 
\begin{equation*}
\begin{cases}
\xi_k(t) &= e^{-i\omega_kt}\xi_k(0), \\
\eta_k(t) &= e^{i\omega_kt}\eta_k(0).
\end{cases}
\end{equation*}
On the other hand, we will show in sec. \ref{nonlinsec} that the nonlinearity $N$ 
belongs to the space $\cH_\rho$.
\end{remark}

\section{The Space of Polynomials} \label{PolSec}

In this section we define a class of polynomials in $\C^\cZ$. 
We first introduce some notations about multi-indices: let $l \geq 2$ and 
$\bj = (j_1,\ldots,j_l) \in \cZ^l$, with $j_i=(k_i,\delta_i)$, we define 
\begin{itemize}
\item the \emph{norm} of the multi-index $\bj$,
\begin{align*}
\bj &:= \max_{i=1,\ldots,l} |j_i| = \max_{i=1,\ldots,l} |k_i|;
\end{align*}
\item the \emph{monomial} associated with $\bj$,
\begin{align*}
z_\bj &:= z_{j_1} \cdots z_{j_l};
\end{align*}
\item the \emph{momentum} of $\bj$,
\begin{align} \label{momentum}
\cM(\bj) &:= k_1\delta_1 + \cdots +k_l\delta_l;
\end{align}
\item the \emph{divisor} associated to $\bj$,
\begin{align} \label{divisor}
\Omega(\bj) &:= \delta_1 \omega_{k_1} + \cdots + \delta_l \omega_{k_l},
\end{align}
where $\omega_k$ is the $k$-th linear frequency of the system, and is given by 
\eqref{freq}.
\end{itemize}

We then define the set of indices with \emph{zero momentum} by 
\begin{align} \label{zeromom}
\cI_l &= \{ \bj = (j_1,\ldots,j_l) \in \cZ^l : \cM(\bj)=0 \}.
\end{align}
We also say that an index $\bj =(j_1,\ldots,j_l) \in \cZ^l$ is \emph{resonant} 
(we denote it by $\bj \in \cN_l$) if $l$ is even and 
$\bj = \boldsymbol{i} \cup \bar{\boldsymbol{i}}$ for some choice of 
$\boldsymbol{i} \in \cZ^{l/2}$. In particular, if $\bj$ is resonant then its 
associated divisor vanishes, namely $\Omega(\bj)=0$, and its associated 
monomial depends only on the actions.
\begin{align*}
z_\bj = z_{j_1} \ldots z_{j_l} &= \xi_{k_1} \eta_{k_1} \cdots \xi_{k_l} \eta_{k_l},
\end{align*}
where for all $k \geq 1$ $I_k(z) = \xi_k\eta_k$ denotes the $k$-th action. 
Finally, we point out that if $z$ is real then $I_k(z) = |\xi_k|^2$; we also 
remark that for odd $l$ the resonant set $\cN_l$ is empty.

\begin{definition}
Let $k \geq 2$, we say that a polynomial $P(z)=\sum a_\bj z_\bj$ belongs to 
$\cP_k$ if $P$ is real, of degree $k$, which has a zero of order at least 2 in 
$z=0$, and if 
\begin{itemize}
\item $P$ contains only monomials having zero momentum, namely such 
that $\cM(\bj)=0$ for $a_\bj \neq 0$, thus $P$ is of the form
\begin{align} \label{polk}
P(z) &= \sum_{l=2}^k \sum_{\bj \in \cI_l} a_\bj z_\bj,
\end{align}
with $a_{\bar\bj}=\bar{a_\bj}$;
\item the coefficients $a_\bj$ are bounded, namely 
\begin{align*}
\sup_{\bj \in \cI_l} |a_\bj| < + \infty, &\; \; \forall l=2,\ldots,k.
\end{align*}
\end{itemize}
\end{definition}

We endow the space $\cP_k$ with the norm 
\begin{align} \label{polknorm}
\norm{ P } &= \sum_{l=2}^k \sup_{\bj \in \cI_l} |a_\bj|.
\end{align}

\begin{remark}
In the following sections we will crucially use the fact that polynomials in 
$\cP_k$ contain only monomial with zero momentum: indeed, this will allow us 
to control the largest index by the others.
\end{remark}

The zero momentum condition is essential to prove the following result.

\begin{proposition} \label{polkprop}
Let $k \geq 2$ and $\rho>0$, then $\cP_k \subset \cH_\rho$. 
Moreover, any homogeneous polynomial $P$ which belongs to $\cP_k$ satisfies 
\begin{align}
|P(z)| &\leq \norm{P} \norm{z}_\rho^k, \; \; \forall z \in \cL_\rho \label{polkest} \\
\norm{X_P(z)}_\rho &\leq 2k \norm{P} \norm{z}_\rho^{k-1}, \; \; \forall z \in \cL_\rho. \label{polkvfest}
\end{align}
Furthermore, if $P \in\cP_k$ and $Q \in\cP_l$, then $\{P,Q\} \in \cP_{k+l-2}$, and
\begin{align} \label{polpoiest}
\norm{ \{P,Q\} } &\leq 2kl \norm{P} \norm{Q}.
\end{align}
\end{proposition}

\begin{proof}
Let 
\begin{align*}
P(z) &= \sum_{\bj \in \cI_k} a_\bj z_\bj,
\end{align*}
then
\begin{align*}
|P(z)| \stackrel{\eqref{polknorm}}{\leq} \norm{P} \sum_{\bj \in \cI_k} |z_{j_1}| \cdots |z_{j_k}| &\leq \norm{P} \norm{z}_{l^1}^k \leq \norm{P} \norm{z}_\rho^k,
\end{align*}
and hence we get \eqref{polkest}.

To prove \eqref{polkvfest}, take $l \in \cZ$, and exploit the zero momentum 
condition in order to get 
\begin{align*}
\left| \dip{P}{z_l} \right| &\leq k \norm{P} \sum_{\stackrel{\bj \in \cZ^{k-1}}{\cM(\bj)=-\cM(l)} } |z_{j_1} \cdots z_{j_l}|.
\end{align*}
We have
\begin{align*}
\norm{X_P(z)}_\rho = \sum_{l \in \cZ} e^{\rho|l|} \left| \dip{P}{z_l} \right| &\leq 
k \norm{P} \sum_{l \in \cZ} \sum_{\stackrel{\bj \in \cZ^{k-1}}{\cM(\bj)=-\cM(l)} } e^{\rho|l|}|z_{j_1} \cdots z_{j_{k-1}}|;
\end{align*}
since $\cM(\bj)=-\cM(l)$,
\begin{align*}
e^{\rho|l|} &\leq \prod_{m=1,\ldots,k-1} e^{\rho|j_m|},
\end{align*}
therefore by summing in $l$ we obtain
\begin{align*}
\norm{X_P(z)}_\rho &\leq 2k \norm{P} \sum_{\bj \in \cZ^{k-1}} e^{\rho|j_1}|z_{j_{k-1}}| \cdots e^{\rho|j_{k-1}}|z_{j_{k-1}}| \leq 2k \norm{P} \norm{z}_\rho^{k-1},
\end{align*}
which gives \eqref{polkvfest}.

Now let us assume that $P$ and $Q$ are homogeneous polynomail of degrees $k$ 
and $l$ respectively, and with coefficients $a_\bk$, $\bk \in \cI_k$, and 
$b_\bl$, $\bl \in \cI_l$. It is easy to check that $\{P,Q\}$ is a monomial of 
degree $k+l-2$ that satisfies the zero momentum condition. Moreover, if we write
\begin{align*}
\{P,Q\}(z) &= \sum_{\bj \in \cI_{k+l-2}} c_\bj z_\bj,
\end{align*}
we have that $c_\bj$ is a sum of coefficients $a_\bk b_\bl$ for which there 
exists $h \geq 1$, and $\delta \in \{\pm1\}$ such that 
\begin{align*}
(h,\delta) &\subseteq \bk \in \cI_k, \\
(h,-\delta) &\subseteq \bl \in \cI_l
\end{align*}
(if for example $(h,\delta)=k_1$ and $(h,-\delta)=l_1$, then 
$(k_2,\ldots,k_k,l_2,\ldots,l_l)=\bj$ ). Hence, for a given $\bj$ the zero 
momentum condition on $\bk$ and $\bl$ determines two possible values of 
$(h,\delta)$. This allow us to deduce \eqref{polpoiest} for monomials; 
its extension to polynomial follows by the definition of the norm 
\eqref{polknorm}.
\end{proof}

\begin{remark}
In Proposition \ref{polkprop} we used a $l^1$-type norm to estimate Fourier 
coefficients and vector fields, instead of the usual $l^2$ norms. 
This choice does not allow to work in Hilbert spaces, and produces a loss of 
regularity each time the estimates are transported from the Fourier space to 
the space of analytic functions $\cA_\rho$. However, by using this approach we 
can control vector fields in a simple way. This argument has been already
used in \cite{faou2013nekhoroshev}.

It may be possible to recover the usual $l^2$ estimates for Fourier 
coefficients and vector fields by introducing a more general space of 
polynomials $\cP_{k_1,k_2}$ ($k_1 \geq 2$, $1 < k_2 \leq \infty$), 
endowed with the following norm
\begin{align*}
\norm{P}_{k_1,k_2} &:= \sum_{l=2}^k \left( \sum_{\bj \in \cI_l} |a_\bj|^{k_2} \right)^{1/k_2}, \; \; 1 < k_2 < + \infty \\
\norm{P}_{k_1,\infty} &:= \norm{P}_{k_1} = \sum_{l=2}^k \sup_{\bj \in \cI_l} |a_\bj|.
\end{align*}
In this case one shoud also introduce a more general space of analytic functions
\begin{align*}
\cL_{\rho,k}:= \{ z \in \C^\cZ : \norm{z}_{\rho,k} := \sum_{j \in \cZ} e^{\rho|j|} |z_j|^k <+\infty \}, \; 1 \leq k < +\infty.
\end{align*}
\end{remark}

\section{Nonlinearity} \label{nonlinsec}

First we recall that the nonlinearity $N$ in \eqref{Nlin} in the coordinates 
$(\psi,\bar\psi)$ takes the form 
\begin{align} \label{Nlin2}
N(\psi,\bar\psi) &= \int_I F \left( \left(\frac{c}{ (c^2-\Delta+ \wtV)^{1/2} }\right)^{1/2} \frac{\psi+\bar\psi}{\sqrt{2}} \right) \di x,
\end{align}
where we assume that $F$ is analytic in a neighbourhood of the origin in 
$\C \times \C$. Therefore there exist positive constants $M$ and $R_0$ 
such that the Taylor expansion of the integrand in the nonlinearity 
\begin{align*}
& F \left( \left(\frac{c}{ (c^2-\Delta+ \wtV)^{1/2} }\right)^{1/2} \frac{\psi+\bar\psi}{\sqrt{2}} \right) \\
&= \sum_{k_1,k_2 \geq 0} \frac{1}{k_1!k_2!} 2^{-(k_1+k_2)/2} \left(\frac{c}{ (c^2-\Delta+ \wtV)^{1/2} }\right)^{(k_1+k_2)/2} \d_{\psi}^{k_1} \d_{\bar\psi}^{k_2} F(0) \psi^{k_1} \bar\psi^{k_2}
\end{align*}
is uniformly convergent on the ball $|\psi|+|\bar\psi| \leq 2R_0$ of 
$\C \times \C$, and is bounded by $M$.

\begin{remark}
The constant $M$ in the above estimate can be chosen uniformly with respect to 
the speed of light $c$, because the smoothing pseudodifferential operator
\begin{align*}
\left(\frac{c}{ (c^2-\Delta+\wtV)^{1/2} }\right)^{1/2}: \cA_\rho \to \cA_\rho
\end{align*}
can be estimated uniformly with respect to $c \geq 1$. 
Indeed it is easy to check that
\begin{align*}
&\norm{ \left(\frac{c}{ (c^2-\Delta+\wtV)^{1/2} }\right)^{1/2} (\psi,\bar\psi) }_\rho \\
&= \sup_{w \in I_\rho} \left| \left(\frac{c}{ (c^2-\Delta+\wtV)^{1/2} }\right)^{1/2} \psi(w) \right| + \left| \left(\frac{c}{ (c^2-\Delta+\wtV)^{1/2} }\right)^{1/2} \bar\psi(w) \right| \\
&= 2 \sup_{w \in I_\rho} \left| \left(\frac{c}{ (c^2-\Delta+\wtV)^{1/2} }\right)^{1/2} \psi(w) \right| \\
&\leq 2 K_\rho \norm{\psi}_\rho,
\end{align*}
 for some constant $K_\rho>0$ that depends only on $\rho$.
\end{remark}

Thus, formula \eqref{Nlin} defines an analytic function on the ball 
$\norm{z}_\rho \leq R_0$ of $\cL_\rho$, and we can write 
\begin{align*}
N(z) &= \sum_{k \geq 0} N_k(z),
\end{align*}
where for all $k \geq 0$ $N_k$ is a homogeneous polynomial defined by 
\begin{align*}
N_k(\xi,\eta) &= \sum_{k_1+k_2=k} \sum_{ (\boldsymbol{a},\boldsymbol{b}) \in \N_0^{k_1} \times \N_0^{k_2} } n_{\boldsymbol{a},\boldsymbol{b}} \xi_{a_1}\cdots\xi_{a_{k_1}} \eta_{b_1}\cdots\eta_{b_{k_2}},
\end{align*}
where 
\begin{align*}
n_{\boldsymbol{a},\boldsymbol{b}} &= \sum_{k_1,k_2 \geq 0} \frac{1}{k_1!k_2!} 2^{-(k_1+k_2)/2} \left(\frac{c}{ (c^2-\Delta+ \wtV)^{1/2} }\right)^{(k_1+k_2)/2} \d_{\psi}^{k_1} \d_{\bar\psi}^{k_2} F(0) \cdot \\
&\cdot \int_I \phi_{a_1}(x) \cdots \phi_{a_{k_1}}(x) \phi_{b_1}(-x) \cdots \phi_{b_{k_2}}(-x) \di x
\end{align*}

Hence it is easy to check that $N_k$ satisfies the zero momentum condition, 
thus $N_k \in \cP_k$ for all $k$, and $\norm{N_k} \leq M R_0^{-k}$.

\begin{remark}
One could also extend this argument by considering not only zero momentum 
monomials, but also monomials with exponentially decreasing or power-law 
decreasing momentum. This would allow to deal with the NLKG equation with a 
multiplicative potential and with nonlinearities that depend on $x$,
\begin{align*}
\frac{1}{c^2} \; u_{tt} \; - \; u_{xx} \; + \; c^2 \; u \; + V(x)u + f(x,u) &= 0, \; \; x \in I,
\end{align*}
but this problem would require a more technical discussion.
\end{remark}

\section{Non resonance conditions} \label{nonresSec}

In order to prove Theorem \ref{NekhThm}, we need to show some 
non resonance properties of the frequencies $\omega=(\omega_k)_{k\geq1}$: 
it will be crucial that these properties hold uniformly 
(or at least, up to a set of small probability) in $(1,+\infty) \times \cV$. 
The argument is similar to the one reported in \cite{pasquali2017longII}. \\

For $r \geq 3$ and $\bj=(j_1,\ldots,j_r) \in \cZ^r$ 
($j_i=(a_i,\delta_i)$, $1 \leq i \leq r$), we define $\mu(\bj)$ 
as the third largest integer among $j_1$,$\ldots$,$j_r$, and we again mention 
that $\bj \in \cZ^r$ is called resonant if $r$ is even and 
$\bj = \boldsymbol{i} \cup \bar{\boldsymbol{i}}$ for some choice of 
$\boldsymbol{i} \in \cZ^{r/2}$.

\begin{theorem} \label{nonrescondthm}
For any sufficiently small $\gamma>0$ there exists a set 
$\cR_\gamma:=\cR_{\gamma,s} \subset \; [1,+\infty) \times \cV$ satisfying 
\begin{align*}
|\cR_{\gamma} \cap ([n,n+1] \times \cV)| &= \cO(\gamma) \; \; \forall n \in \N_0,
\end{align*}
and $\exists \tau>0$ such that 
$\forall (c,(v_j)_j) \in ([1,+\infty) \times \cV) \setminus \cR_{\gamma}$, 
for any sufficiently large $N \geq 1$ and for any $r \geq 1$ 
\begin{align} \label{nonrescond}
|\Omega(\bj)| &\geq \frac{\gamma}{\mu(\bj)^{\tau r}},
\end{align}
for any non resonant $\bj \in\cZ^{r+2} \setminus \cN^{r+2}$ with $\mu(\bj)\leq N$.
\end{theorem}

\begin{remark}
The non resonance condition in Theorem \ref{nonrescondthm} is slightly weaker 
from the one proved in sec. 2.4 of \cite{faou2013nekhoroshev}; 
this is the reason for which Faou and Grébert are able to prove a result which 
is valid for a set of potentials of full measure, while our result is valid 
only for ``most'' of the values of speed of light and potentials.

We also mention that a non resonance condition similar to our one 
has been already used in \cite{bambusi2003birkhoff} (see (3.3)) 
and in \cite{bambusi2006birkhoff} (see (r-NR)).
\end{remark}

Actually, instead of proving that condition \eqref{nonrescond} is satisfied, 
we will prove a more general result. 
We use the notation ``$a \sleq b$'' (resp. ``$a \sgeq b$'') to mean 
``there exists a constant $K>0$ independent of $c$ such that $a \leq Kb$'' 
(resp. $a \geq Kb$). We also write $N(\bj) := \prod_{i=1}^r (1+|j_i|)$. 

\begin{remark}
Observe that $\mu(\bj) \leq |\bj| \leq N(\bj)$. 
Furthermore, $N(\bj) \leq N$ implies that $|\bj| \leq N$, since 
$N(\bj) \geq |j_1|+\ldots+|j_r| \geq |\bj|$. On the other hand, 
$|\bj| \leq N$ implies that $N(\bj) = \prod_{i=1}^r (1+|j_i|) \leq (1+N)^r$.
\end{remark}

\begin{proposition} \label{cond0pot}
Let $c\geq1$ be fixed. 
Then for any sufficiently small $\gamma>0$ $\exists \cV'_{s,M,\gamma} \subset \cV$ 
with $|\cV \setminus \cV'_{s,M,\gamma}|=\cO(\gamma)$, and 
$\exists \tau>1$ such that $\forall (v_j)_{j \geq 1} \in \cV'_{s,M,\gamma}$, 
for any $N\geq1$ and for any $r \geq 1$ 
\begin{align} \label{nonres0prop}
|\Omega(\bj) \; + \; m| &\geq \frac{\gamma}{N(\bj)^\tau}
\end{align}
$\forall$ $m \in \mathbb{Z}$, 
for any $\bj \in \cZ^{r} \setminus \cN^{r}$ with $|\bj| \leq N$.
\end{proposition}

\begin{proof}
Fix $r \geq 1$, $\bj \in \cZ^{r}$ non resonant and $m \in \Z$.

Let $p_{\bj,m}((v_j)_{j\geq1}) \; := \; \sum_{j=1}^r \omega_{a_j}\delta_j + m$, 
and let $ 1 \leq h \leq r$. Then 
\begin{align*}
\left| \frac{\d p_{\bj,m}}{\d v'_h} \right| = \left| \frac{\delta_h a_h^{-s}}{2 \sqrt{1+\lambda_{a_h}/c^2}} \right| &\succeq \; \frac{1}{2 a_h^s \sqrt{1+a_h^{\max(s,2)}} } \geq \frac{1}{2 N(\bj)^s \sqrt{1+N(\bj)^{\max(s,2)}}} \; > \; 0,
\end{align*}
since $a_r \leq \prod_{i=1}^r (1+a_i) = \prod_{i=1}^r (1+|j_i|) = N(\bj)$. 
Hence, if we define for $\gamma_0>0$ the set 
\begin{align*}
Res_{\bj,m,r}(\gamma_0) &:= \{(v'_j)_{j\geq1} \; : \; |p_{\bj,m}((v_j)_{j\geq1})| < \gamma_0\},
\end{align*}
we have by Lemma 17.2 of \cite{russmann2001invariant}
\begin{align*}
\left| Res_{\bj,m,r}(\gamma_0)  \right| &\leq  \gamma_0 \; N(\bj)^{s+\max(s,2)/2}.
\end{align*}
Now, we estimate the measure of $\bigcup_{r,m,\bj} Res_{\bj,m,r}(\gamma_0)$; note 
that the argument in the modulus in \eqref{nonres0prop} can be small only if 
$|m| \leq rN$. Therefore
\begin{align}
\left| \bigcup_{r \geq 1} \bigcup_{|m| \leq rN} \; \bigcup_{|\bj| \leq N} \; \{(v'_j)_{j\geq1} \; : \; |p_{\bj,m}((v_j)_{j\geq1})| < \gamma_0\} \right| &\leq \sum_{r=1}^\infty \gamma_0 \; N^{r+1}(1+N)^{r(s+\max(s,2)/2)} \label{estresset}, 
\end{align}
since $N(\bj) \stackrel{|\bj| \leq N}{\leq} (1+N)^r$, and by choosing 
$\gamma_0 \; = \; \frac{\gamma}{(1+N)^{\tau r}}$ with 
$\tau > s+2+\frac{\max(s,2)}{2}$  we get the thesis.
\end{proof}

\begin{proposition} \label{cond1mel}
For any sufficiently small $\gamma>0$ there exists 
a set $\cR_\gamma:=\cR_{\gamma,s,M} \subset \; [1,+\infty) \times \cV$ satisfying 
\begin{align} \label{estrresset}
|\cR_{\gamma} \cap ([n,n+1] \times \cV)| &= \cO(\gamma) \; \; \forall n \in \N_0,
\end{align}
and $\exists \tau>1$ such that 
$\forall (c,(v_j)_j) \in ([1,+\infty) \times \cV) \setminus \cR_{\gamma}$, 
for any sufficiently large $N \geq 1$ and for any $r \geq 1$ 
\begin{align} \label{nonres1prop}
\left|\Omega(\bj) \; + \; \sigma \omega_l\right| &\geq \frac{\gamma}{N(\bj)^\tau}
\end{align}
for any $\bj \in \cZ^{r}$ with $|\bj| \leq N$, $\sigma=\pm 1$, $l \geq N$ 
such that $(\bj,(l,\sigma)) \in \cZ^{r+1} \setminus \cN^{r+1}$ is non resonant.
\end{proposition}

\begin{proof}
Without loss of generality, we can choose $\sigma=-1$.

Now fix $r \geq 1$, a non resonant $\bj \in \cZ^{r}$ with $|\bj| \leq N$, and 
$l \geq N$. Set $p_{\bj,l}(c,(v_j)_{j\geq1}) \; := \; \Omega(\bj)-\omega_l$. 
We can rewrite the function $p_{\bj,l}$ in the following way: 
\begin{align*}
p_{\bj,l}(c,(v_j)_{j\geq1}) \; &= \alpha c^2 \; + \; \sum_{j=1}^r \frac{\delta_j\lambda_{a_j}}{1+\sqrt{1+\lambda_{a_j}/c^2}} \; - \frac{\lambda_l}{1+\sqrt{1+\lambda_l/c^2}},
\end{align*}
where $\alpha:= (\sum_{j=1}^r \delta_j)-1 \in \{-r-1,\ldots,r-1\}$. 
Now we have to distinguish some cases: 

\emph{Case $\alpha=0$}: in this case we have that $p_{\bj,l}$ can be small 
only if $l^2 \leq 3(N^2+N^s)^2r^2$. So to obtain the result we just apply 
Proposition \ref{cond0pot} with $N':=\sqrt{3}(N^2+N^s)r$, $r'=r+1$. 

\emph{Case $\alpha\neq0$, $c \leq \lambda_N^{1/2}r^{1/2}$}: we have that 
\begin{align*}
\sum_{j=1}^r \; c\sqrt{c^2+\lambda_{a_j}} \delta_j  &\leq  r \sqrt{c^4+c^2\lambda_N} \; \leq \; \sqrt{2} \; r^2 \lambda_N,
\end{align*}
so $|\Omega(\bj)-\omega_l|$ can be small only for $l^2<rN^2$. 
Therefore, in order to get the thesis we apply Proposition \ref{cond0pot} 
with $N':=\sqrt{r}N$, $r':=r+1$. 

\emph{Case $\alpha > 0$, $c > \lambda_N^{1/2}r^{1/2}$}: first notice that 
if we set $f(x):=\frac{x^2}{2\sqrt{1+x}(1+\sqrt{1+x})^2 }$, and we put 
$x_j:=\lambda_{a_j}/c^2$, in this regime we get 
\begin{align*}
\left|\sum_{j=1}^r \delta_jf(x_j)\right| &\leq \frac{r}{2} f(x_N) \leq \frac{1}{2}.
\end{align*}
Now define $\tilde{p}_{\bj,l}(c^2) := \alpha c^2 \; - \; \frac{\lambda_l}{1+\sqrt{1+\lambda_l/c^2}}$. One can verify that 
\begin{align*}
\tilde{p}_{\bj,l}(c^2) &=  0; \\
c^2=c^2_{l,\alpha} &:= \frac{\lambda_l}{\alpha(\alpha+2)},
\end{align*}
and that
\begin{align*}
\frac{\d \tilde{p}_{\bj,l}}{\d (c^2)}(c^2_{l,\alpha}) &= 
 \alpha - \frac{\alpha^2(\alpha+2)^2}{2\sqrt{1+\alpha(\alpha+2)}(1+\sqrt{1+\alpha(\alpha+2)})^2}  > 0.
\end{align*}
Besides, there exists a sufficiently small $\rho>0$ such that in the interval 
$\left[ c^2_{l,\alpha}-\frac{\varrho}{\alpha(\alpha+2)},c^2_{l,\alpha}+\frac{\varrho}{\alpha(\alpha+2)} \right] \; =: \; [c^2_{l,\alpha,-},c^2_{l,\alpha,+}]$ we have  
\begin{align*}
\frac{\d \tilde{p}_{\bj,l}}{\d (c^2)}(c^2) &> \left(\frac{1}{2}+\frac{1}{2(r+1)}\right) \alpha 
\end{align*}
(we can also assume that the intervals $([c^2_{l,\alpha,-},c^2_{l,\alpha,+}])_{l \geq N}$ are pairwise disjoint, by taking $\rho$ sufficiently small). 
Thus, by exploiting Lemma 17.2 of \cite{russmann2001invariant}, we get   
\begin{align} \label{estreset1}
\left| \left\{ c^2 \in B\left(c^2_{l,\alpha},\frac{\varrho}{\alpha(\alpha+2)}\right] : |\tilde{p}_{\bj,l}(c^2)| \leq \gamma_0 \right\}\right| &\leq \gamma_0 \frac{2(r+1)}{(r+2)\alpha} \leq \frac{2 \gamma_0}{\alpha}
\end{align}
for any $\gamma_0>0$ s.t. $\frac{2\gamma_0}{\alpha} < \frac{\varrho}{\alpha(\alpha+2)}$; $\gamma_0 < \frac{\varrho}{2 (\alpha+2)}$. 

Now, since in this regime 
$\left|\frac{\d (p_{\bj,l}-\tilde{p}_{\bj,l})}{\d c^2}\right| \leq \frac{1}{2}$, 
we can deal with $p_{\bj,l}$ as in \eqref{estreset1}, and we get that 
for each $n \in \N_0$
\begin{align} 
\left| \bigcup_{l\geq N} \; \left( \{c^2 \in [c^2_{l,\alpha,-},c^2_{l,\alpha,+}] \; : |p_{\bj,l}(c^2)| \leq\gamma_0\} \cap [n,n+1] \right) \right| &\leq K_\alpha \gamma_0,
\end{align}
for some $K_\alpha>0$. By taking the union over $|\bj| \leq N$ and arguing as in 
Proposition \ref{cond0pot}, we can deduce \eqref{estresset}.

\emph{Case $\alpha<0$, $c > \lambda_N^{1/2}r^{1/2}$}: since
\begin{align*}
\left| \sum_{j=1}^r \frac{\delta_j\lambda_{a_j}}{1+\sqrt{1+\lambda_{a_j}/c^2}} \right| &\leq \frac{r\lambda_N}{2} \leq \frac{c^2}{2},
\end{align*}
we have that $p_{k,l}$ can be small only if $\lambda_N < r \lambda_N$. 
So, in order to get the result, we apply Proposition \ref{cond0pot} 
with $N':=r^{1/2}N$, $r':=r+1$. 
\end{proof}

\begin{theorem} \label{cond2mel}
For any sufficiently small $\gamma>0$ there exists 
a set $\cR_\gamma:=\cR_{\gamma,s,r} \subset \; [1,+\infty) \times \cV$ satisfying 
\begin{align*}
|\cR_{\gamma} \cap ([n,n+1] \times \cV)| &= \cO(\gamma) \; \; \forall n \in \N_0,
\end{align*}
and $\exists \tau>1$ such that 
$\forall (c,(v_j)_j) \in ([1,+\infty) \times \cV) \setminus \cR_{\gamma}$, 
for any sufficiently large $N \geq 1$ and for any $r \geq 1$ 
\begin{align} \label{nonres2prop}
\left|\Omega(\bj) \; + \; \sigma_1 \omega_l \; + \; \sigma_2 \omega_m \right| &\geq \frac{\gamma}{N(\bj)^\tau}
\end{align}
for any $\bj \in \cZ^{r}$ with $|\bj| \leq N$, 
$\sigma_1,\sigma_2=\pm 1$, $m > l \geq N$ such that 
$(\bj,(l,\sigma_1),(m,\sigma_2)) \in \cZ^{r+2}\setminus\cN^{r+2}$ is non resonant.
\end{theorem}

\begin{proof}
If $\sigma_i=0$ for $i=1,2$, then we can conclude by using Proposition 
\ref{cond1mel}. \\
\indent Now fix $r \geq 1$, a non resonant $\bj \in \cZ^{r}$ with $|\bj|\leq N$,
two positive integers $l$ and $m$ such that $m > l \geq N$, 
and assume that $\sigma_1=-1$, $\sigma_2=1$. Introduce   
\begin{align*}
p_{\bj,l,m}(c^2)  &:= \Omega(\bj) \; - \omega_l(c^2) \; +  \omega_m(c^2).
\end{align*}
Now fix $\delta>3$. If $m \sleq N^\delta$, then we can conclude by applying 
Proposition \ref{cond0pot} and \ref{cond1mel}. 
So from now on we will assume that $m,l>N^\delta$. 

We have to distinguish several cases: 

\emph{Case $c < \lambda_l^\alpha$:} we point out that, since 
\begin{align*}
c \sqrt{c^2+\lambda_l} = c \lambda_l^{1/2} \sqrt{1+\frac{c^2}{\lambda_l}} &=  c \lambda_l^{1/2} \left( 1+ \frac{c^2}{2\lambda_l} + O\left(\frac{1}{\lambda_l^2}\right) \right),
\end{align*}
we get (denote $m=l+j$)
\begin{align*}
\omega_m-\omega_l &=  j c \; + \; \frac{1}{2}\left(\frac{v_m}{m}-\frac{v_l}{l}\right) \; + \; \frac{c^3}{2\lambda_l^{1/2}} - \frac{c^3}{2\lambda_m^{1/2}} + O\left(\frac{1}{m^3}\right) + O\left(\frac{1}{l^3}\right),
\end{align*}
that is, the integer multiples of c are accumulation points for the differences 
between the frequencies as $l,m \to \infty$, provided that $\alpha<\frac{1}{6}$.

\emph{Case $c > \lambda_m$:} in this case we have 
(again by denoting $m=l+j$) that 
$\lambda_m-\lambda_l = j(j+2l)+(v_m-v_l) \; = \; 2jl+j^2+a_{lm}$, with $|a_{lm}| \leq \frac{C}{l}$, so that
\begin{align*}
p_{\bj,l,m} &= \Omega(\bj) \; \pm 2jl  \pm j^2 \pm a_{lm}.
\end{align*}
If $l > 2 \,C N^{\tau}/\gamma$ then the term $a_{lm}$ represents a negligible 
correction and therefore we can conclude by applying Proposition \ref{cond0pot}.
 On the other hand, if $l \leq 2 \, C N^{\tau}/\gamma$, we can apply the same 
Proposition with $N':= 2 C N^{\tau}/\gamma$ and $r':=r+2$. 

\emph{Case $\lambda_l^{1/6} \leq c \sleq \lambda_l^{1/2}$:} 
if we rewrite the quantity to estimate 
\begin{align*}
p_{\bj,l,m}(c^2) &= \alpha c^2 \; + \; \sum_{h=1}^r \frac{\lambda_{a_h}\delta_h}{1+\sqrt{1+\frac{\lambda_{a_h}}{c^2} }} \; + \; \omega_m \; - \; \omega_l,
\end{align*}
where $\alpha:=\sum_{h=1}^r \delta_h$, we distinguish three cases:
\begin{itemize}
\item if $\alpha>0$, then we notice that
\begin{align*}
\left| \sum_{h=1}^r \frac{\lambda_{a_h}\delta_h}{1+\sqrt{1+\frac{\lambda_{a_h}}{c^2} }} \right| \leq \frac{r \lambda_N}{1+\sqrt{1+\lambda_N/c^2}} &\leq \frac{r\lambda_N}{1+\sqrt{1+\lambda_N/\lambda_l}} \leq \frac{r\lambda_N}{2},
\end{align*}
\begin{align*}
|\omega_m-\omega_l| &= c \frac{\lambda_m-\lambda_l}{\sqrt{c^2+\lambda_m}+\sqrt{c^2+\lambda_l}} \stackrel{m>l}{\geq} \frac{c\lambda_l^{1/2}}{\sqrt{c^2+\lambda_m}+\sqrt{c^2+\lambda_l}} \\
&\gtrsim \frac{N^{\delta/3} \lambda^{1/2}}{\sqrt{N^{2\delta/3}+\lambda_m^{1/2}}+\sqrt{N^{2\delta/3}+\lambda_l^{1/2}}} >0,
\end{align*}
thus $|p_{\bj,l,m}|>|\lambda_l^{1/3}-\frac{r}{2}\lambda_N|>0$, since $l>N^3$; \\
\item if $\alpha=0$, then we just notice that 
\begin{align*}
|\omega_m-\omega_l| \geq \gamma(\lambda_m-\lambda_l) &\stackrel{m>l}{\gtrsim} \; \gamma_0 \; \lambda_l^{1/2},
\end{align*}
which is greater than $\gamma_0/N^\tau$ for $\tau>-1$, since $l>N^3$; \\
\item if $\alpha<0$, then we just recall that 
$|\omega_m-\omega_l|>\gamma_0\lambda_l^{1/2}$, and by choosing $\gamma_0$ 
sufficiently small (actually $\gamma_0\leq|\alpha|$) we get that 
also in this case $p_{\bj,l,m}$ is bounded away from zero.
\end{itemize}

\end{proof}

Now it is easy to deduce Theorem \ref{nonrescondthm} by exploiting Theorem 
\ref{cond2mel}.

\section{Normal form} \label{BNFSec}

Fix $N \geq 1$. For a fixed integer $k \geq 3$, we define
\begin{align*}
\sI_k(N) &:= \{ \bj \in \cI_k : \mu(\bj)>N \}.
\end{align*}

\begin{definition} \label{NFdef}
Let $N$ be an integer. We say that a polynomial $Z \in \cP_k$ is in $N$-normal 
form if it is of the form
\begin{align*}
Z(z) &= \sum_{l=3}^k \sum_{\bj \in \cN_l \cup \sI_l(N)} a_\bj z_\bj,
\end{align*}
namely, if $Z$ contains either monomials that depend only of the actions or 
monomials whose index $\bj$ satisfies $\mu(\bj)>N$, i.e. it involves 
at least three modes with index greater than $N$.
\end{definition}

The reason for introducing such a definition of normal form is given by the 
observation that the vector field of a monomial of the form $z_{j_1}\ldots z_{j_k}$ containing at least three modes with index larger than $N$ induces a flow 
whose dynamics can be controlled for exponentially long ($N$-dependent) time 
scales. This will prevent exchanges of energy between low- and high-index modes 
for such time scales. 

In \cite{bambusi2003birkhoff} and \cite{bambusi2006birkhoff} such monomials 
were neglected, since the contribution of their vector fields was 
to be small in Sobolev norms, and this will keep all the modes (almost) 
invariant. The point is that, even though the contribution of the vector 
fields of such monomials is not necessarily small in analytic norm, it can be 
controlled for exponentially long times. This key property was already used 
by Faou, Grébert and Paturel in \cite{faou2010birkhoff2} and by 
Faou and Grébert in \cite{faou2013nekhoroshev}. \\

Before we state and prove the aforementioned property we just state an 
elementary lemma.

\begin{lemma} \label{ODElemma}
Let $f:\R \to \R_+$ be a continuous funciton, and let $y: \R \to \R_+$ 
be a differentiable function such that 
\begin{align*}
\frac{\di}{\di t} y(t) &\leq 2 f(t) \sqrt{y(t)}, \; \; \forall t \in \R.
\end{align*}
Then
\begin{align*}
\sqrt{y(t)} &\leq \sqrt{y(0)} + \int_0^t f(s) \di s, \; \; \forall t \in \R.
\end{align*}
\end{lemma}

For a given $N$ and $z \in \cL_\rho$, we set $\cR^N_\rho(z):= \sum_{|j| \geq N} e^{\rho|j|}|z_j|$. Observe that if $z \in \cL_{\rho+\mu}$, then
\begin{align} \label{estrem}
\cR^N_\rho(z) &\leq e^{-N \mu} \norm{z}_{\rho+\mu}.
\end{align}

\begin{proposition}
Let $N \in \N_0$, and $k \geq 3$. 
Let $Z$ be a homogeneous polynomial of degree $k$ in $N$-normal form. 
Denote by $z(t)$ the real solution of the flow associated to the Hamiltonian 
$H_0+Z$. Then 
\begin{align}
\cR^N_\rho(t) &\leq \cR^N_\rho(0) + 4k^3 \norm{Z} \int_0^t \cR^N_\rho(z(s))^2 \norm{z(s)}_\rho^{k-3} \di s, \label{estanrem} \\
\norm{z(t)}_\rho &\leq \norm{z(0)}_\rho + 4k^3 \norm{Z} \int_0^t \cR^N_\rho(z(s))^2 \norm{z(s)}_\rho^{k-3} \di s. \label{estandata}
\end{align}
\end{proposition}

\begin{proof}
Let $a \geq 1$ be fixed, and let $I_a(t)=\xi_a(t)\eta_a(t)$ be the actions 
associated to the solution of the Hamiltonian $H_0+Z$. 
Recalling that $H_0=H_0(I)$, and writing $\bj=(j_1,\ldots,j_k) \in \cZ^k$ 
($j_i=(a_i,\delta_i)$, $1 \leq i \leq k$), we have
\begin{align*}
|e^{2\rho a} \dot I_a| = |e^{2\rho a} \{I_a,Z\}| \stackrel{\eqref{polpoiest}}{\leq} 
2k \norm{Z} |e^{\rho a} \sqrt{I_a}| \left( \sum_{\substack{\cM(\bj)=\pm a, \\ 2 \; \text{indices} \; > N} } e^{\rho a} |z_{j_1} \ldots z_{j_{k-1}}| \right),
\end{align*}
and by applying Lemma \ref{ODElemma} we obtain
\begin{align} \label{estanact}
e^{\rho a} \sqrt{I_a(t)} \leq e^{\rho a} \sqrt{I_a(0)} + 2k \norm{Z} 
\int_0^t \left( \sum_{\substack{\cM(\bj)=\pm a, \\ 2 \; \text{indices} \; > N}  e^{\rho a_1}} |z_{j_1} \ldots e^{\rho a_{k-1}} z_{j_{k-1}}| \right) \di s.
\end{align}
Ordering the multi-indices in such a way that $a_1$ and $a_2$ are the largest, 
and recalling that $z(t)$ is real, we obtain after a summation in $a>N$
\begin{align*}
\cR^N_\rho(z(t) &\leq \cR^N_\rho(z(0)) + 4k^3 \norm{Z} \int_0^t \left( \sum_{\substack{|j_1|,j_2| \geq N, \\ j_3,\ldots,j_k \in \cZ} } e^{\rho a_1} |z_{j_1} \ldots e^{\rho a_{k-1}} z_{j_{k-1}}| \right) \di s \\
&\leq \cR^N_\rho(z(0)) + 4k^3 \norm{Z} \int_0^t \cR^N_\rho(z(s))^2 \norm{z(s)}_\rho^{k-3} \di s.
\end{align*}
\end{proof}

\begin{remark} \label{expestrem}
Estimates \eqref{estanrem}-\eqref{estandata} will be fundamental for proving 
Theorem \ref{NekhThm}. Indeed, take $z(t)$ the solution of the Hamiltonian 
in $N$-normal form with corresponding initial datum $z_0$, and assume that 
$\norm{z_0}_\rho=R$. Hence, since $\cR^N_\rho(z_0)=\cO(R e^{-N\rho})$, estimates 
\eqref{estanrem}-\eqref{estandata} ensure that $\cR^N_\rho(z(t))$ remains of 
order $\cO(R e^{-N\rho})$ and that the norm of $z(t)$ remains of order $\cO(R)$ 
up to times $t$ of order $\cO(e^{N\rho})$.
\end{remark}

Next we exploit the non resonance condition \eqref{nonrescond} and the 
definition of normal forms to estimate the solution of a homological equation.

\begin{proposition} \label{homeqprop}
Assume that the non resonance condition \eqref{nonrescond} is fulfilled. 
Let $N$ be fixed, and let $Q$ be a homogeneous polynomial of degree $k$. 
Then the homological equation
\begin{align} \label{homeq}
\{\chi,H_0\} - Z &= Q
\end{align}
admits a polynomial solution $(\chi,Z)$ homogeneous of degree $k$ such that 
$Z$ is is $N$-normal form, and such that 
\begin{align} \label{homeqest}
\norm{Z}\leq \norm{Q}, \; &\; \norm{\chi} \leq \frac{N^{\tau k}}{\gamma} \norm{Q}.
\end{align}
\end{proposition}

\begin{proof}
Assume that $Q= \sum_{\bj \in \cI_k} Q_\bj z_\bj$, we look for 
$Z= \sum_{\bj \in \cI_k} Z_\bj z_\bj$ and $\chi= \sum_{\bj \in \cI_k} \chi_\bj z_\bj$ 
such that the homological equation \eqref{homeq} is satisfied.
Then Equation \eqref{homeq} can be rewritten in term of polynomial coefficients 
\begin{align*}
i \Omega(\bj) \chi_\bj - Z_\bj &= Q_\bj, \; \; j \in \cI_k,
\end{align*}
where $\Omega(\bj)$ is defined as in \eqref{divisor}. By setting
\begin{equation*}
\begin{cases}
Z_\bj = Q_\bj \text{ and } \chi_\bj=0, & \text{ if } \; \; j \in \cN_k \text{ or } \mu(\bj)>N, \\
Z_\bj = 0 \text{ and } \chi_\bj=\frac{Q_\bj}{i\Omega(\bj)}, & \text{ if } j \notin \cN_k \text{ or } \mu(\bj) \leq N,
\end{cases}
\end{equation*}
and by exploiting \eqref{nonrescond}, we can deduce \eqref{homeqest}.
\end{proof}

\section{Proof of the main Theorem}

We now prove Theorem \ref{NekhThm} by exploiting the results from the 
previous sections. The argument is inspired by the one in \cite{faou2013nekhoroshev}.

\subsection{Recursive equation}

In this subsection we want to construct a canonical transformation $\cT$ such 
that the Hamiltonian $(H_0+N) \circ \cT$ is in $N$-normal form, up to a small 
remainder term. We use Lie transform in order to generate the transformation 
$\cT$: hence we look for polynomials $\chi=\sum_{k\geq3}^r \chi_k$ and 
$Z=\sum_{k\geq3}^r Z_k$ in $N$-normal form and a smooth Hamiltonian $R$ such that
$\di^\alpha R(0)=0$ for all $\alpha\in\N^\cZ$ with $|\alpha|\geq r$, and such that
\begin{align} \label{HamTr}
(H_0+N) \circ \Phi^1_\chi &= H_0+Z+R.
\end{align}
The exponential estimate will be obtained by optimizing the choice of $r$ and $N$. \\

We recall that if $\chi$ and $K$ are two Hamiltonian, we have that for all 
$k \geq 0$
\begin{align*}
\frac{\di^k}{\di t^k} (K \circ \Phi^t_\chi) &= \{\chi,\ldots \{\chi,K\} \ldots \}(\Phi^t_\chi) = (ad_\chi^k K)(\Phi^t_\chi),
\end{align*}
where $ad_\chi K:=\{\chi,K\}$. On the other hand, if $K$ and $L$ are homogeneous 
polynomial of degree respectively $k$ and $l$, than $\{K,L\}$ is a homogeneous 
polynomial of degree $k+l-2$. Thus, by using Taylor formula
\begin{align} \label{TayRem}
(H_0+N) \circ \Phi^1_\chi - (H_0+N) &= \sum_{k=0}^{r-3} \frac{1}{(k+1)!} ad_\chi^k(\{\chi,H_0+N\}) + \cO_r,
\end{align}
where by ``$+ \cO_r$'' we mean ``up to a smooth function $R$ satisfying $\d^\alpha R(0)=0$ for all $\alpha \in \N^\cZ$ with $|\alpha| \geq r$. Now, recall the 
following relation
\begin{align*}
\left( \sum_{k=0}^{r-3} \frac{B_k}{k!} \zeta^k \right) \left( \sum_{k=0}^{r-3} \frac{1}{(k+1)!} \zeta^k \right) &= 1 + \cO(|\xi|^{r-2}),
\end{align*}
where $B_k$ are the Bernoulli numbers defined by the expansion of the 
generating function $\frac{z}{e^z-1}$. We just recall that there exists 
$K>0$ such that $|B_k| \leq k! \; K^k$ for all $k$. \\

Hence, defining the two differential operators 
\begin{align*}
\cA_r &:= \sum_{k=0}^{r-3} \frac{1}{(k+1)!} ad_\chi^k, \; \\
\cB_r &:= \sum_{k=0}^{r-3} \frac{B_k}{k!} ad_\chi^k,
\end{align*}
we obtain 
\begin{align*}
\cB_r\cA_r &= id + \cC_r,
\end{align*}
where $\cC_r$ is a differential operator satisfying $\cC_r\cO_3=\cO_r$. 
Applying $\cB_r$ to both sides of Eq. \eqref{TayRem} we get 
\begin{align*}
\{\chi,H_0+N\} &= \cB_r (Z-N) +\cO_r.
\end{align*}
Plugging the decomposition in homogeneous polynomials of $\chi$, $Z$ and $N$ 
in the last equation and comparing the terms with the same degree, we have 
the following recursive equations
\begin{align} \label{receq}
\{\chi_m,H_0\} - Z_m &= Q_m, \; \; m=3,\ldots,r,
\end{align}
where
\begin{align}
Q_m &= - N_m - \sum_{k=3}^{m-1} \{\chi_k,N_{m+2-k}\} \nonumber \\
&+ \sum_{k=1}^{m-3} \frac{B_k}{k!} \sum_{ \substack{l_1+\ldots+l_{k+1}=m+2k \\ 3 \leq l_i \leq m-k} } ad_{\chi_{l_1}} \ldots ad_{\chi_{l_k}}(Z_{l_{k+1}}-N_{l_{k+1}}). \label{mterm}
\end{align}

We point out that in \eqref{mterm} the condition $l_i \leq m-k$ is a 
consequence of $l_i \geq 3$ and $l_1+\ldots+l_{k+1}=m+2k$. 
Once these recursive equations are solved, we can define the remainder term 
$R=(H_0+N) \circ \Phi^1_\chi - (H_0+Z)$. By construction we have that $R$ is 
analytic on a neighbourhood of the origin in $\cL_\rho$, and that $R=\cO_r$. 
Hence, by Taylor formula
\begin{align}
R &= \sum_{m \geq r+1} \sum_{k=1}^{m-3} \frac{1}{k!} \sum_{ \substack{l_1+\ldots+l_k=m+2k \\ 3 \leq l_i \leq m-k} }  ad_{\chi_{l_1}} \ldots ad_{\chi_{l_k}} H_0 \nonumber \\
&+ \sum_{m \geq r+1} \sum_{k=0}^{m-3} \frac{1}{k!} \sum_{ \substack{l_1+\ldots+l_{k+1}=m+2k \\ 3 \leq l_1+\ldots+l_k \leq r \\ 3 \leq l_{k+1}} }  ad_{\chi_{l_1}} \ldots ad_{\chi_{l_k}} N_{l_{k+1}}. \label{remBNF}
\end{align}

\begin{lemma} \label{itlemma}
Assume that the non resonance condition \eqref{nonrescond} is satisfied. 
Let $r$ and $N$ be fixed. 
For $m=3,\ldots,r$ there exists homogeneous polynomials $\chi_m$ and $Z_m$ 
of degree $m$ that solve \eqref{receq}, with $Z_m$ in $N$-normal form and 
such that 
\begin{align} \label{itestm}
\norm{ \chi_m } + \norm{ Z_m } &\leq (K m N^\tau)^{m^2},
\end{align}
where $K$ is a positive constant that does not depend on $r$ or $N$.
\end{lemma}

\begin{proof}
We define $\chi_m$ and $Z_m$ by induction, via Proposition \ref{homeqprop}. 
Note that \eqref{itestm} is trivially satisfied for $m=3$, provided that $K$ is 
sufficiently large. Estimate \eqref{homeqest}, combined with \eqref{polpoiest} 
and with the classical estimate on the Bernoulli numbers, we have that for all 
$m \geq 3$
\begin{align*}
\frac{\gamma}{N^{\tau r}} \norm{\chi_m} + \norm{Z_m} &\stackrel{\eqref{mterm}}{\leq} \norm{N_m} + 2\sum_{k=3}^{m-1} k(m+2-k) \norm{N_{m+2-k}} \norm{\chi_k} \\
&+ 2 \sum_{k=1}^{m-3} (K m)^k \sum_{ \substack{l_1+\ldots+l_{k+1}=m+2k \\ 3 \leq l_i \leq m-k} } l_1 \norm{\chi_{l_1}} \ldots l_k \norm{\chi_{l_k}} \norm{ Z_{l_{k+1}}-N_{l_{k+1}} },
\end{align*}
for some constant $K>0$. Now we set $\beta_m:=m(\norm{\chi_m}+\norm{Z_m})$; 
using the fact that $\norm{N_m} \leq M/R_0^m$ (see the end of Sec. 
\ref{nonlinsec}), we get
\begin{align*}
\beta_m &\leq \beta_m^{(1)} + \beta_m^{(2)}, \\
\beta_m^{(1)} &:= (K N^\tau)^m \; m^3 \sum_{k=3}^{m-1} \beta_k, \\
\beta_m^{(2)} &:= N^{\tau m} (K m)^{m-1} \sum_{k=1}^{m-3} \sum_{ \substack{l_1+\ldots+l_{k+1}=m+2k \\ 3 \leq l_i \leq m-k} } \beta_{l_1} \ldots \beta_{l_k} (\beta_{l_{k+1}} + \norm{N_{l_{k+1}}}),
\end{align*}
where $K$ depends on $M$, $R_0$, $\gamma$. We have to prove by recurrence 
that $\beta_m \leq (K m N^\tau)^{m^2}$, for $m \geq 3$: this is trivially true 
for $m=3$, by choosing a sufficiently large $K$. Now assume that 
$\beta_j \leq (K j N^\tau)^{j^2}$ for $j=3,\ldots,m-1$; we get 
\begin{align*}
\beta_m^{(1)} \leq (K N^\tau)^m m^4 (K m N^\tau)^{(m-1)^2} &\leq (K m N^\tau)^{m^2-m+1} \leq \frac{1}{2} (K m N^\tau)^{m^2},
\end{align*}
for $m \geq 4$, provided that $K>2$. 
On the other hand, since $\norm{N_m} \leq M/R_0^m$, we can assume that 
$\norm{N_{l_{k+1}}} \leq \beta_{l_{k+1}}$, hence 
\begin{align*}
\beta_m^{(2)} \leq N^{\tau m} (K m)^{m-1} \sum_{k=1}^{m-3} \sum_{ \substack{l_1+\ldots+l_{k+1}=m+2k \\ 3 \leq l_i \leq m-k} } (K N^\tau (m-k))^{l_1^2+\ldots+l_{k+1}^2}.
\end{align*}
Observe that the maximum of $l_1^2+\ldots+l_{k+1}^2$ when $l_1+\ldots+l_{k+1}=m+2k$ and $3 \leq l_i \leq m-k$ is obtained for $l_1=\ldots=l_k=3$, $l_{k+1}=m-k$, and 
that its value at the maximum is $(m-k)^2+9k$. Furhermore, 
$|\{ (l_1,\ldots l_k) \in \cZ^k : l_1+\ldots+l_{k+1}=m+2k, 3 \leq l_i \leq m-k \}|$ is smaller than $m^{k+1}$. Therefore
\begin{align*}
\beta_m^{(2)} &\leq \frac{1}{2} (K m N^\tau)^{m^2}
\end{align*}
for any $m \geq 4$, and for a sufficiently large $K$.
\end{proof}

\subsection{Normal Form Theorem}

Given a positive $R_0$, we set $B_\rho(R_0):=\{ z \in \cL_\rho : \norm{z}_\rho < R_0\}$.

\begin{theorem} \label{BNFthm}
Let $N$ be analytic on a ball $B_\rho(R_0)$ for some $R_0>0$ and $\rho>0$. 
Fix $\beta < 1$ and $M>1$, 
and assume that the non resonance condition \eqref{nonrescond} is fulfilled.
Then there exist a sufficiently small $\epsilon_0>0$ 
and a positive $\sigma>0$ such that the following holds: 
for all $\epsilon < \epsilon_0$ there exist a polynomial $\chi$, 
a polynomial $Z$ in $|\log \epsilon|^{1+\beta}$-normal form, and 
a Hamiltonian $R$ analytic on $B_\rho(M\epsilon)$ such that
\begin{align} \label{TrHam}
(H_0+N) \circ \Phi^1_\chi &= H_0+Z+R,
\end{align}
and such that for all $z \in B_\rho(M \epsilon)$
\begin{align}
\norm{X_Z(z)}_\rho + \norm{X_\chi(z)}_\rho &\leq 2\epsilon^{3/2}, \label{BNFest1} \\
\norm{X_R(z)}_\rho&\leq \epsilon e^{-\frac{1}{4} |\log \epsilon|^{1+\beta}}. \label{BNFest2}
\end{align}
\end{theorem}

\begin{proof}
Using Lemma \ref{itlemma}, we can construct for all $r$ and $N$ polynomial 
Hamiltonians
\begin{align*}
\chi(z) &= \sum_{k=3}^r \chi_k(z), \\
Z(z) &= \sum_{k=3}^r Z_k(z),
\end{align*}
with $Z$ in $N$-normal form such that \eqref{TrHam} holds with $R=\cO_r$. 
Now fix $\epsilon$, and choose 
\begin{align} \label{Nrchoice}
N:=N(\epsilon)=|\log\epsilon|^{1+\beta},\; &\; r:=r(\epsilon)=|\log\epsilon|^\beta.
\end{align}
\eqref{Nrchoice} is motivated by the fact that we can control the error 
induced by $Z$ by Remark \ref{expestrem}, while the error induced by $R$ 
can be estimated by Lemma \ref{itlemma}. Indeed, in this way we have
\begin{align}
\norm{\chi_k} &\stackrel{\eqref{itestm}}{\leq} (K k N^\tau)^{k^2} \nonumber \\
&\leq \exp( k(\tau k(1+\beta) \log|\log\epsilon| + k \log(K k)) ) \nonumber \\
&\stackrel{k \leq r}{\leq} \exp( k(\tau r(1+\beta) \log|\log\epsilon| + r \log(K r)) ) \nonumber \\
&\stackrel{\eqref{Nrchoice}}{\leq} \exp( k|\log\epsilon| (\tau |\log\epsilon|^{\beta-1} (1+\beta) \log|\log\epsilon| + |\log\epsilon|^{\beta-1} \log(K|\log\epsilon|^\beta)) ) \nonumber \\
&\stackrel{\beta<1}{\leq} \epsilon^{-k/8}, \label{estchik}
\end{align}
provided that $\epsilon_0$ is sufficiently small. 
Hence, by Proposition \ref{polkprop} we get that for all $z\in B_\rho(M\epsilon)$
\begin{align*}
|\chi_k(z)| &\leq \epsilon^{-k/8} (M\epsilon)^k = M^k \epsilon^{7k/8},
\end{align*}
and we can deduce that
\begin{align*}
|\chi(z)| &\leq \sum_{k \geq 3} M^k \epsilon^{7k/8} \leq \epsilon^{3/2},
\end{align*}
provided that $\epsilon_0$ is sufficiently small. Similarly, we have 
\begin{align*}
\norm{X_{\chi_k(z)}}_\rho &\leq 2k \epsilon^{-k/8} (M\epsilon)^{k-1} = 2kM^{k-1} \epsilon^{7k/8 - 1}, \; \; 3 \leq k \leq r, \\
\norm{X_\chi(z)}_\rho &\leq \sum_{k \geq 3}2kM^{k-1} \epsilon^{7k/8 - 1} \sleq  \epsilon^{-1} \epsilon^{21/8} \leq \epsilon^{3/2},
\end{align*}
provided that $\epsilon_0$ is sufficiently small. Similar estimates hold also 
for $Z$, and therefore we can deduce \eqref{BNFest1}. \\

In order to estimate the remainder, we recall that by \eqref{receq} 
$ad_{\chi_{l_k}}H_0 = Z_{l_k} + Q_{l_k}$, therefore by \eqref{mterm} we obtain 
\begin{align*}
\norm{ad_{\chi_{l_k}}H_0} &\leq (K k N^\tau)^{l_k^2} \stackrel{\eqref{estchik}}{\leq} \epsilon^{-l_k/8}.
\end{align*}
Thus, we can exploit \eqref{remBNF} and the fact that 
$\norm{N_{l_{k+1}}} \leq M/R_0^{l_{k+1}}$ in order to get
\begin{align*}
\norm{X_R(z)}_\rho &\stackrel{\eqref{remBNF}}{\leq} \sum_{m \geq r+1} \sum_{k=0}^{m-3} m (K r)^{3m} \epsilon^{-\frac{m+2k}{8}} \epsilon^{m-1} \\
&\leq \sum_{m \geq r+1} m^2 (K r)^{3m} \epsilon^{m/2} \\
&\leq (K r)^{3r} \epsilon^{r/2} \\
&\stackrel{\eqref{Nrchoice}}{\leq} \epsilon e^{-\frac{1}{4} |\log\epsilon|^{1+\beta} },
\end{align*}
provided that $\epsilon_0$ is sufficiently small.
\end{proof}

Now we can prove Theorem \ref{NekhThm} by applying Theorem \ref{BNFthm}.

\begin{proof}[Proof of Thm. \ref{NekhThm}]
Let $(\psi_0,\bar\psi_0) \in \cA_{\rho}$, $|(\psi_0,\bar\psi_0)|_\rho=R$, and 
denote by $z(0)$ the corresponding coefficients, which belong by Lemma 
\ref{anlemma} to $\cL_{3\rho/4}$, and satisfy
\begin{align*}
\norm{z(0)}_{3\rho/4} &\leq \frac{K_\rho}{4} R,
\end{align*}
where $K_\rho := \frac{2^3}{1-e^{-\rho/4}}$. 

Let $z(t)$ be the local solution in $\cL_{\rho/2}$ of the Hamiltonian system 
asoociated to $H_0+N$. Let $\chi$, $Z$ and $R$ be given by Theorem \ref{BNFthm} 
with $M=K_\rho$, and let $y(t)=\Phi^1_\chi(z(t))$. We recall that 
$\chi(z)=\cO(\norm{z}^3)$, that $\Phi^1_\chi$ is close to the identity, 
$\Phi^1_\chi(z)=z+\cO(\norm{z}^2)$: hence, for sufficiently small $R$ we get 
\begin{align*}
\norm{y(0)}_{3\rho/4} &\leq \frac{K_\rho}{2} R,
\end{align*}
and in particular
\begin{align*}
\cR^N_\rho(y(0)) &\leq \frac{K_\rho}{2} R e^{-\frac{\rho}{4} N} \leq \frac{K_\rho}{2} R e^{-\sigma N},
\end{align*}
where $\sigma=\sigma_\rho \leq \frac{\rho}{4}$. 

Now let $T_R$ be the maximal time $T$ such that 
\begin{align*}
\cR^N_\rho(y(t)) &\leq K_\rho R e^{-\sigma N}, \; \; |t| \leq T, \\
\norm{y(t)}_\rho &\leq K_\rho R,  \; \; |t| \leq T.
\end{align*}
By construction we have 
\begin{align*}
y(t) &= y(0) + \int_0^t X_{H_0+Z}(y(s)) \di s + \int_0^t X_R(y(s)) \di s,
\end{align*}
hence by using \eqref{estandata} for the second term and \eqref{BNFest2} 
for the third term, we obtain that for $|t| < T_R$
\begin{align}
\cR^N_\rho(y(t)) &\leq \frac{1}{2} K_\rho R e^{-\sigma N} 
+ 4|t| \sum_{k=3}^r \norm{Z_k} k^3 (K_\rho R)^{k-1} e^{-2\sigma N} 
+ |t| R e^{-\frac{1}{4} |\log\epsilon|^{1+\beta}} \nonumber \\
&\leq \left( \frac{1}{2} 
+ 4|t| \sum_{k=3}^r \norm{Z_k} k^3 (K_\rho R)^{k-2} e^{-\sigma N} 
+ |t| R e^{-\frac{1}{8} |\log\epsilon|^{1+\beta}}
\right) K_\rho R e^{-\sigma N}, \label{estrem2}
\end{align}
where in the last inequality we have used that $\sigma=\min(1/8,\rho/4)$ 
and that $N=|\log\epsilon|^{1+\beta}$. Using Lemma \ref{itlemma}, we get 
\begin{align*}
\cR^N_\rho(y(t)) &\leq \left( \frac{1}{2} + K|t| R e^{-\sigma N} \right) K_\rho R e^{-\sigma N},
\end{align*}
and thus, for sufficiently small $R$,
\begin{align}
\cR^N_\rho(y(t)) &\leq K_\rho R e^{-\sigma N}, \; \; |t| \leq \min(T_\rho,e^{\sigma N}), \label{estremfinal}
\end{align}
and similarly
\begin{align}
\norm{y(t)}_\rho &\leq K_\rho R, \; \; |t| \leq \min(T_\rho,e^{\sigma N}). \label{estsolfinal}
\end{align}

Now, observe that by \eqref{estremfinal} and \eqref{estsolfinal} we have 
$T_R \geq e^{\sigma N}$; in particular, we have
\begin{align*}
\norm{z(t)}_\rho &\leq 2K_\rho R, \; \; |t| \leq e^{\sigma N}=e^{ -\sigma |\log\epsilon|^{1+\beta} }.
\end{align*}
Using Lemma \ref{anlemma}, we obtain \eqref{expest1} by setting 
$K = \frac{2^5}{(1-e^{-\rho/4})^2}$. \\

Estimate \eqref{expest2} is a consequence of Theorem \ref{BNFthm} and 
\eqref{estandata}: indeed, it just suffices to remark that $z(t)$ is 
$R^2$-close to $y(t)$, which in turn is almost invariant, since
\begin{align*}
\sum_{k \geq 1} e^{\rho|k|} ||y_k(t)|-|y_k(0)|| &\stackrel{\eqref{estandata}}{\leq} 
4 |t| \sum_{k=3}^r \norm{Z_k} k^3 (K_\rho R)^{k-1} e^{-2\sigma N} + |t| R e^{-\frac{1}{4} |\log\epsilon|^{1+\beta} },
\end{align*}
and arguing as in \eqref{estrem2} we can finally deduce \eqref{expest2}.
\end{proof}

\bibliography{P_NLKG_Nek2017}
\bibliographystyle{alpha}

\end{document}